\documentclass{article}
\usepackage{amscd}
\usepackage{amsmath}
\usepackage{amssymb}
\usepackage{amsthm}
\usepackage{ascmac}
\usepackage{bm}
\usepackage{graphicx}
\usepackage{here}
\usepackage[pdftex]{hyperref}%
\usepackage{mathtools}\numberwithin{equation}{section}
\usepackage{minitoc}
\usepackage{multicol}
\usepackage{multirow}
\usepackage{setspace}
\theoremstyle{definition}
\newtheorem{df}{Definition}

\newtheorem{eg}{Example}
\newtheorem{rmk}{Remark}
\theoremstyle{theorem}

\newtheorem{thm}{Theorem}

\newtheorem{lem}{Lemma}

\title{Horizontal miniatures and normal-sized miniatures\\
of convex lattice polytopes}
\author{Takashi HIROTSU}
\date{\today}
\begin{document}
\maketitle
\begin{abstract}
Let $n,$ $d,$ and $r$ be integers such that $0 \leq r \leq d \leq n,$ and let $P \subset \mathbb R^n$ be a $d$-dimensional convex lattice polytope. 
In this article, we prove that the ratio of the $r$-dimensional volume of a normal-sized miniature of $P$ to that of $P$ is given by $1:\binom{d+r+1}{r},$ which generalizes the author's previous results on the volumes of the unit hypercube and lattice simplices in the case where $r = d = n.$ 
This theorem is proven by establishing that the number of horizontal miniatures of $P$ with resolution $t$ is a polynomial of degree $d+1$ in $t$ whose leading coefficient is $\mathrm{vol}\,P/(d+1),$ which is derived from Ehrhart theory.
\end{abstract}
\section{Introduction}
Let $n$ and $r$ be integers such that $0 \leq r \leq n.$ 
For each polytope $P \subset \mathbb R^n,$ let $\mathrm{vol}_r(P)$ denote the $r$-dimensional volume of $P.$ 
We recall the following definition.
\begin{df}[{\cite[Definitions 1--3 and 5]{Hir25}}]\label{df-mini}
Let $P \subset \mathbb R^n$ be a $d$-dimensional polytope.
\begin{enumerate}
\item[(1)]
We call a polytope $M \subset \mathbb R^n$ a {\itshape miniature} of $P$ if $M$ is contained in $P$ and similar to $P,$ allowing a scale factor of $0.$ 
In particular, when the scale factor is $0,$ the miniature degenerates into a single point, which is referred to as a {\itshape degenerate} miniature.
\item[(2)]
A miniature $M$ of $P$ is said to be {\itshape horizontal} if $M$ is transformed into $P$ by translation and nonnegative integral rescaling.
\end{enumerate}\par
Let $M$ be a miniature of $P,$ and let $t > 0$ be an integer.
\begin{enumerate}
\item[(3)]
We call $M$ a miniature of $P$ with {\itshape resolution} $t$ if all vertices of $M$ lie in $(t^{-1}\mathbb Z)^n.$
\item[(4)]
Let $\mathcal H_t(P)$ (resp.~$\mathcal H_t^+(P)$) denote the set of all horizontal miniatures (resp.~nondegenerate horizontal miniatures) of $P$ with resolution $t.$ 
If the limit 
\[\mu_{r}^{\mathrm{nl}}(P) := \lim_{t \to \infty}\frac{\sum_{M \in \mathcal H_t^+(P)}\mathrm{vol}_r(M)}{\#\mathcal H_t^+(P)}\] 
exists in $\mathbb R,$ then $P$ is said to be {\itshape $\mu_{r}^{\mathrm{nl}}$-measurable}. 
If $P$ is $\mu_{d}^{\mathrm{nl}}$-measurable, then a miniature of $P$ with volume $\mu_{d}^{\mathrm{nl}}(P)$ is called a {\itshape normal-sized miniature} of $P.$
\end{enumerate}
\end{df}
\begin{rmk}
Definition \ref{df-mini} updates the definition of a miniature introduced in the author's previous work \cite{Hir25} by including the degenerate case.
\end{rmk}
The following theorems have been previously established by the author.
\begin{thm}[{\cite[Theorem 2]{Hir22}}]\label{thm-nl-cube}
Let $P = [0,1]^d.$ 
Then $P$ is $\mu_{d}^{\mathrm{nl}}$-measurable and 
\[\mu_{d}^{\mathrm{nl}}(P) = \binom{2d+1}{d}^{-1}.\]
\end{thm}
\begin{thm}[{\cite[Theorem 3]{Hir25}}]\label{thm-nl-simplex}
Let $P \subset \mathbb R^d$ be a $d$-dimensional lattice simplex. 
Then $P$ is $\mu_{d}^{\mathrm{nl}}$-measurable and 
\[\mu_{d}^{\mathrm{nl}}(P) = \binom{2d+1}{d}^{-1}\mathrm{vol}_d(P).\]
\end{thm}
The main results of this article are as follows, which are proven in Section \ref{sec-nl-conv}.
\begin{thm}\label{thm-nl-conv-r}
Let $P \subset \mathbb R^n$ be a $d$-dimensional convex lattice polytope, and let $r \leq d.$ 
Then $P$ is $\mu_{r}^{\mathrm{nl}}$-measurable and 
\begin{equation} 
\mu_{r}^{\mathrm{nl}}(P) = \binom{d+r+1}{r}^{-1}\mathrm{vol}_r(P). \label{eq-nl-conv-r} 
\end{equation} 
\end{thm}
\begin{rmk}
The formula \eqref{eq-nl-conv-r} is also valid for the case where $d < r,$ since $\mathrm{vol}_r(M) = \mathrm{vol}_r(P) = 0$ for any horizontal miniature $M$ of $P$ in this case.
\end{rmk}
The first few values of the binomial coefficient $\binom{d+r+1}{r}$ appearing in \eqref{eq-nl-conv-r} are given in Table \ref{tbl-binom}.
\begin{table}[h]
\centering
\begin{tabular}{c|ccccccc} 
$d \backslash r$ & 0 & 1 & 2 & 3 & 4 & 5 & 6 \\ \hline 
$0$ & $1$ \\ 
$1$ & $1$ & $3$ \\ 
$2$ & $1$ & $4$ & $10$ \\ 
$3$ & $1$ & $5$ & $15$ & $35$ \\ 
$4$ & $1$ & $6$ & $21$ & $56$ & $126$ \\ 
$5$ & $1$ & $7$ & $28$ & $84$ & $210$ & $462$ \\ 
$6$ & $1$ & $8$ & $36$ & $120$ & $330$ & $792$ & $1716$ 
\end{tabular}
\caption{The values of $\binom{d+r+1}{r}$ for $0 \leq r \leq d \leq 6.$}\label{tbl-binom}
\end{table}\par
Theorem \ref{thm-nl-conv-r} yields the following results.
\begin{thm}\label{thm-nl-inv}
Let $P \subset \mathbb R^n$ be a $d$-dimensional convex lattice polytope, and let $r \leq d.$ 
Then $\mu_{r}^{\mathrm{nl}}(P)$ is invariant under lattice equivalence.
\end{thm}
\begin{thm}\label{thm-nl-pie-r}
Let $P,$ $Q \subset \mathbb R^n$ be $d$-dimensional convex lattice polytopes, and let $r \leq d.$ 
Suppose that $P\cup Q$ and $P\cap Q$ are also convex lattice polytopes. 
Then $\mu_{r}^{\mathrm{nl}}$ satisfies the valuation property, namely, 
\begin{equation} 
\mu_{r}^{\mathrm{nl}}(P\cup Q) = \mu_{r}^{\mathrm{nl}}(P)+\mu_{r}^{\mathrm{nl}}(Q)-\mu_{r}^{\mathrm{nl}}(P\cap Q). \label{eq-nl-pie-r} 
\end{equation} 
\end{thm}
To prove Theorem \ref{thm-nl-conv-r}, we utilize Ehrhart theory. 
For each bounded convex set $S \subset \mathbb R^n$ and integer $t \geq 0,$ the {\itshape lattice-point enumerator} $\mathcal E(S;t)$ for the $t$-th dilate $tS$ of $S$ is defined as 
\[\mathcal E(S;t) := \# (tS\cap \mathbb Z^n).\] 
\begin{thm}[{\cite[Theorem 3.8 and Corollaries 3.15, 3.17, and 3.20]{BR15}}]\label{thm-ehr-poly}
Let $P \subset \mathbb R^n$ be a $d$-dimensional convex lattice polytope. 
Then $\mathcal E(P;t)$ is a polynomial in $t$ of the form 
\[\mathcal E(P;t) = c_dt^d+c_{d-1}t^{d-1}+\cdots +c_1t+c_0,\] 
where $d!c_i \in \mathbb Z$ for any $i \in \{ 0,1,\dots,d\},$ $c_d = \mathrm{vol}_d(P),$ and $c_0 = 1.$
\end{thm}
The polynomial $\mathcal E(P;t)$ in Theorem \ref{thm-ehr-poly} is referred to as the {\itshape Ehrhart polynomial} of $P.$ 
We use the following definition and theorem.
\begin{df}
Let $P \subset \mathbb R^n$ be a convex lattice polytope, and let $t \geq 0$ be an integer.
\begin{enumerate}
\item[(1)]
We call a polytope $G \subset \mathbb R^n$ a {\itshape giganture} of $P$ in $tP$ if $G$ is contained in $tP$ and similar to $P,$ allowing a scale factor of $0.$ 
In particular, when the scale factor is $0,$ the giganture degenerates into a single point, which is referred to as a {\itshape degenerate} giganture.
\item[(2)]
A giganture $G$ of $P$ is said to be {\itshape horizontal} if $P$ is transformed into $G$ by translation and nonnegative integral rescaling (or equivalently, if $G = iP+a$ for some integer $i \geq 0$ and some vector $a \in \mathbb Z^n$).
\item[(3)]
Let $\mathcal H(P;t)$ (resp.~$\mathcal H^+(P;t)$) denote the number of horizontal gigantures (resp.~nondegenerate horizontal gigantures) of $P$ in $tP.$
\end{enumerate}
\end{df}
\begin{thm}\label{thm-nl-no}
Let $P \subset \mathbb R^n$ be a $d$-dimensional convex lattice polytope, and let $t > 0$ be an integer.
\begin{enumerate}
\item[\textup{(1)}]
For any $i \in \{ 0,1,\dots,t\},$ the number of horizontal miniatures of $P$ with resolution $t$ and scale factor $i/t,$ and the number of horizontal gigantures of $P$ in $tP$ with scale factor $i$ are both equal to $\mathcal E(P;t-i).$ 
\item[\textup{(2)}]
Let $V \subset \mathbb Z^n$ be the set of all vertices of $P,$ and let $\mathrm{Pyr}\,P$ denote the convex hull of $\{ (v,0) \mid v \in V\}\cup\{ (0,\dots,0,1)\}$ in $\mathbb R^{n+1}.$ 
Then $\#\mathcal H_t(P)$ and $\mathcal H(P;t)$ agree as polynomials in $t$ and are given by 
\begin{align} 
\#\mathcal H_t(P) &= \mathcal H(P;t) \notag \\ 
&= \sum_{i = 0}^{t}\mathcal E(P;i) \notag \\ 
&= \mathcal E(\mathrm{Pyr}\,P;t) \notag \\ 
&= a_{d+1}t^{d+1}+a_{d}t^{d}+\cdots +a_1t+a_0, \label{eq-nl-no} 
\end{align} 
where $(d+1)!a_i \in \mathbb Z$ for any $i \in \{ 0,1,\dots,d+1\},$ $a_{d+1} = \mathrm{vol}_d(P)/(d+1),$ and $a_0 = 1.$
\end{enumerate}
\end{thm}
\begin{rmk}
Theorem \ref{thm-nl-no} yields 
\[\#\mathcal H_t^+(P) = \mathcal H^+(P;t) = \mathcal H(P;t)-\mathcal E(P;t) = \sum_{i = 0}^{t-1}\mathcal E(P;i) = \mathcal E(\mathrm{Pyr}\,P;t-1).\]
\end{rmk}
\section{Horizontal Miniatures of Convex Lattice Polytopes}\label{sec-nl-no}
In this section, we prove Theorem \ref{thm-nl-no}. 
\begin{proof}[Proof of Theorem \ref{thm-nl-no}]
\begin{enumerate}
\item[(1)]
By definition, these two numbers trivially coincide. 
Let $a \in \mathbb Z^n.$ 
By the convexity of $P,$ the embedding condition $iP+a \subset tP$ is equivalent to the condition that the translation vector $a$ lies in the Minkowski difference $tP-iP = (t-i)P.$ 
Therefore, the desired numbers are equal to 
\[\#\{ a \in \mathbb Z^n \mid iP+a \subset tP\} = \# ((t-i)P\cap\mathbb Z^n) = \mathcal E(P;t-i).\] 
\item[(2)]
By definition, we have $\#\mathcal H_t(P) = \mathcal H(P;t).$ 
The assertion (1) yields 
\[\mathcal H(P;t) = \sum_{i = 0}^{t}\mathcal E(P;t-i) = \sum_{i = 0}^{t}\mathcal E(P;i) = \mathcal E(\mathrm{Pyr}\,P;t),\] 
which can be written in the form \eqref{eq-nl-no}, since $\mathrm{Pyr}\,P$ is a $(d+1)$-dimensional convex lattice polytope with volume $\mathrm{vol}_d(P)/(d+1)$ (see Figure \ref{fig}).
\end{enumerate}
\end{proof}
\begin{figure}[h]
\centering
\includegraphics{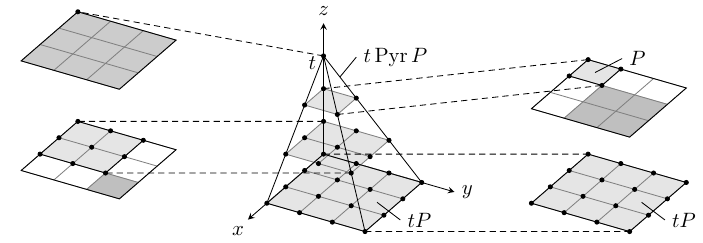}
\caption{The number $\mathcal H(P;t)$ in the case where $P = [0,1]^2$ and $t = 3.$}\label{fig}
\end{figure}
\begin{rmk}
The leading coefficient of $\mathcal H^+(P;t) = \mathcal E(\mathrm{Pyr}\,P;t-1)$ is also $\mathrm{vol}\,P/(d+1)$ and its constant term vanishes, as 
\[\mathcal E(\mathrm{Pyr}\,P;-1) = (-1)^{d+1}\mathcal E((\mathrm{Pyr}\,P)^\circ;1) = (-1)^{d+1}\# ((\mathrm{Pyr}\,P)^\circ\cap\mathbb Z^{n+1}) = 0\] 
by the Ehrhart--Macdonald reciprocity law (see \cite[Theorem 4.1]{BR15}).
\end{rmk}
\begin{eg}
\begin{enumerate}
\item[(1)]
If $P \subset \mathbb R^2$ is the triangle with vertices $(0,0),$ $(1,0),$ and $(0,1),$ then $\mathcal H(P;t)$ is given by 
\[\mathcal H(P;t) = \frac{(t+1)(t+2)(t+3)}{6} = \frac{1}{6}t^3+t^2+\frac{11}{6}t+1,\] 
which coincides with the $(t+1)$-th triangular pyramidal number.
\item[(2)]
If $P \subset \mathbb R^2$ is the unit square $[0,1]^2,$ then $\mathcal H(P;t)$ is given by 
\[\mathcal H(P;t) = \frac{(t+1)(t+2)(2t+3)}{6} = \frac{1}{3}t^3+\frac{3}{2}t^2+\frac{13}{6}t+1,\] 
which coincides with the $(t+1)$-th square pyramidal number.
\end{enumerate}
\end{eg}
\section{Normal-sized Miniatures of Convex Lattice Polytopes}\label{sec-nl-conv}
In this section, we prove Theorems \ref{thm-nl-conv-r}--\ref{thm-nl-pie-r}. 
We use the following lemma.
\begin{lem}\label{lem-sum-prod}
Let $p,$ $q > 0$ be integers. 
Then, for any integer $t > 0,$ we have 
\[\sum_{i = 0}^{t-1}i^p(t-i)^q = \sum_{i = 1}^{t}i^p(t-i)^q = \frac{p!q!}{(p+q+1)!}t^{p+q+1}+O(t^{p+q}),\] 
where $O$ denotes the Landau big-O symbol.
\end{lem}
\begin{proof}
By the standard Riemann sum approximation, we have 
\begin{align*} 
\sum_{i = 0}^{t-1}i^p(t-i)^q &= \sum_{i = 1}^{t}i^p(t-i)^q \\ 
&= t^{p+q+1}\cdot\frac{1}{t}\sum_{i = 1}^{t}\left(\frac{i}{t}\right) ^p\left( 1-\frac{i}{t}\right) ^q \\ 
&= t^{p+q+1}\left(\int_{0}^{1}x^p(1-x)^q\,dx+O(t^{-1})\right) \\ 
&= t^{p+q+1}(B(p+1,q+1)+O(t^{-1})) \\ 
&= \frac{p!q!}{(p+q+1)!}t^{p+q+1}+O(t^{p+q}), 
\end{align*} 
where $B(x,y)$ denotes the beta function.
\end{proof}
\begin{proof}[Proof of Theorem \ref{thm-nl-conv-r}]
Theorems \ref{thm-ehr-poly} and \ref{thm-nl-no}(2) imply 
\[\#\mathcal H_t^+(P) = \mathcal H^+(P;t) = \mathcal H(P;t)-\mathcal E(P;t) = \frac{\mathrm{vol}_d(P)}{d+1}t^{d+1}+O(t^d).\] 
Furthermore, Theorems \ref{thm-nl-no}(1) and \ref{thm-ehr-poly}, and Lemma \ref{lem-sum-prod} imply 
\begin{align*} 
\sum_{M \in \mathcal H_t^+(P)}\mathrm{vol}_r(M) &= \sum_{i = 1}^{t}\mathcal E(P;t-i)\left(\frac{i}{t}\right) ^{r}\mathrm{vol}_r(P) \\ 
&= \frac{\mathrm{vol}_r(P)}{t^r}\sum_{i = 1}^{t}i^{r}\mathcal E(P;t-i) \\ 
&= \frac{\mathrm{vol}_r(P)}{t^r}\sum_{i = 0}^{t-1}(t-i)^{r}\mathcal E(P;i) \\ 
&= \frac{\mathrm{vol}_d(P)\mathrm{vol}_r(P)}{t^r}\sum_{i = 0}^{t-1}(t-i)^r(i^d+O(i^{d-1})) \\ 
&= \frac{\mathrm{vol}_d(P)\mathrm{vol}_r(P)}{t^r}\left(\frac{d!r!}{(d+r+1)!}t^{d+r+1}+O(t^{d+r})\right) \\ 
&= \frac{d!r!}{(d+r+1)!}\mathrm{vol}_d(P)\mathrm{vol}_r(P)t^{d+1}+O(t^d). 
\end{align*} 
Combining these two asymptotic expansions yields 
\[\frac{\sum_{M \in \mathcal H_t^+(P)}\mathrm{vol}_r(M)}{\#\mathcal H_t^+(P)} \xrightarrow[t \to \infty]{} \frac{(d+1)!r!}{(d+r+1)!}\mathrm{vol}_r(P) = \binom{d+r+1}{r}^{-1}\mathrm{vol}_r(P).\] 
This completes the proof.
\end{proof}
\begin{rmk}
Alternatively, even if we define the volume limit by including the degenerate miniatures, i.e., by using $\mathcal H(P;t)$ instead of $\mathcal H^+(P;t),$ the resulting limit coincides with $\mu_r^{\mathrm{nl}}(P).$ 
This is because the measures of the single points in the numerator vanish, and the number of degenerate miniatures in the denominator grows only as a polynomial of degree $d,$ which is asymptotically dominated by the growth of degree $d+1$ of the nondegenerate ones.
\end{rmk}
\begin{proof}[Proof of Theorem \ref{thm-nl-inv}]
Combining Theorem \ref{thm-nl-conv-r} with the fact that $\mathrm{vol}_r(P)$ is invariant under lattice equivalence, we obtain the desired result.
\end{proof}
\begin{proof}[Proof of Theorem \ref{thm-nl-pie-r}]
By combining Theorem \ref{thm-nl-conv-r} with the valuation property for $\mathrm{vol}_r$, namely, 
\[\mathrm{vol}_r(P\cup Q) = \mathrm{vol}_r(P)+\mathrm{vol}_r(Q)-\mathrm{vol}_r(P\cap Q),\] 
we conclude that $\mu_{r}^{\mathrm{nl}}$ satisfies \eqref{eq-nl-pie-r}.
\end{proof}
Finally, we provide a few examples of the values of $\mu_{r}^{\mathrm{nl}}(P)$ in Theorem \ref{thm-nl-conv-r}.
\begin{eg}
\begin{enumerate}
\item[(1)]
Let $P \subset \mathbb R^2$ be the triangle with vertices $(0,0),$ $(1,0),$ and $(0,1).$ 
Then 
\begin{align*} 
\mu_{2}^{\mathrm{nl}}(P) &= \binom{5}{2}^{-1}\mathrm{vol}_2(P) = \frac{1}{10}\cdot\frac{1}{2} = \frac{1}{20} 
\intertext{and} 
\mu_{1}^{\mathrm{nl}}(P) &= \binom{4}{1}^{-1}\mathrm{vol}_1(P) = \frac{2+\sqrt 2}{4}. 
\end{align*} 
\item[(2)]
Let $P \subset \mathbb R^3$ be the tetrahedron with vertices $(0,0,0),$ $(1,0,0),$ $(0,1,0),$ and $(0,0,1).$ 
Then 
\begin{align*} 
\mu_{3}^{\mathrm{nl}}(P) &= \binom{7}{3}^{-1}\mathrm{vol}_3(P) = \frac{1}{35}\cdot\frac{1}{6} = \frac{1}{210}, \\ 
\mu_{2}^{\mathrm{nl}}(P) &= \binom{6}{2}^{-1}\mathrm{vol}_2(P) = \frac{1}{15}\cdot\frac{3+\sqrt 3}{2} = \frac{3+\sqrt 3}{30}, 
\intertext{and} 
\mu_{1}^{\mathrm{nl}}(P) &= \binom{5}{1}^{-1}\mathrm{vol}_1(P) = \frac{3+3\sqrt 2}{5}. 
\end{align*} 
\end{enumerate}
\end{eg}
\begin{eg}
Let $P = [0,1]^d.$ 
Since the total $r$-dimensional volume of the $r$-dimensional faces of $P$ is given by $\mathrm{vol}_r(P) = \binom{d}{r}2^{d-r}$ (see \cite[Section 4.4]{Gru03}), we have 
\begin{align*} 
\mu_{r}^{\mathrm{nl}}(P) &= \binom{d+r+1}{r}^{-1}\binom{d}{r}2^{d-r} \\ 
&= \frac{(d+1)!r!}{(d+r+1)!}\cdot\frac{d!}{r!(d-r)!}2^{d-r} \\ 
&= \frac{d!(d+1)!}{(d-r)!(d+r+1)!}2^{d-r}. 
\end{align*} 
\end{eg}


\begin{thebibliography}{9}
\bibitem{BR15}
M.~Beck and S.~Robins, {\itshape Computing the Continuous Discretely: Integer Point Enumeration in Polyhedra}, 2nd ed., Undergraduate Texts in Mathematics, Springer, New York, 2015. 
\bibitem{Gru03}
B.~Gr\"{u}nbaum, {\itshape Convex Polytopes}, 2nd ed., Graduate Texts in Mathematics, vol.~221, Springer, New York, 2003.
\bibitem{Hir22}
T.~Hirotsu, Normal-sized hypercuboids in a given hypercube, preprint, arXiv:2211.15342.
\bibitem{Hir25}
T.~Hirotsu, Average-sized miniatures and normal-sized miniatures of lattice polytopes, preprint, arXiv:2501.00459.
\end{thebibliography}
\end{document}